\documentclass[11pt,reqno]{amsart}

\numberwithin{equation}{section}
\usepackage[all,cmtip]{xy}
\usepackage{amssymb,rotating,amsmath,amsfonts,amsthm,color,graphics,epsfig,bbm,bm,manfnt}

\newcommand{\calD}{\mathcal{D}}
\newcommand{\calE}{\mathcal{E}}

\newcommand{\calT}{\mathcal{T}}

\newcommand{\mA}{\mathbb{A}}

\newcommand{\mC}{\mathbb{C}}

\newcommand{\mN}{\mathbb{N}}

\newcommand{\mR}{\mathbb{R}}

\newcommand{\mT}{\mathbb{T}}

\newcommand{\mZ}{\mathbb{Z}}

\newcommand{\bba}{\mathbf{a}}
\newcommand{\bbb}{\mathbf{b}}

\newcommand{\bbe}{\mathbf{e}}

\newcommand{\bbh}{\mathbf{h}}

\newcommand{\bbn}{\mathbf{n}}

\newcommand{\bbu}{\mathbf{u}}
\newcommand{\bbv}{\mathbf{v}}

\newcommand{\bbx}{\mathbf{x}}
\newcommand{\bby}{\mathbf{y}}
\newcommand{\bbz}{\mathbf{z}}

\newcommand{\bbA}{\mathbf{A}}
\newcommand{\bbB}{\mathbf{B}}
\newcommand{\bbC}{\mathbf{C}}

\newcommand{\balpha}{\bm{\alpha}}

\newtheorem{theorem}{Theorem}[section]

\newtheorem{corollary}[theorem]{Corollary}
\newtheorem{proposition}[theorem]{Proposition}

\theoremstyle{definition}

\newtheorem{remark}[theorem]{Remark}

\newcommand{\nm}{\,\rule[-.6ex]{.13em}{2.3ex}\,}
\theoremstyle{definition}
\newtheorem{definition}[theorem]{Definition}

\theoremstyle{definition}

\begin{document}

\keywords{periodic distributions, Fourier series, Bass stable rank,
  topological stable rank, B\'ezout equation}

\subjclass{Primary 46F05; Secondary 42B05, 55N15, 18F25}

\title[On the ring of periodic distributions]{On the B\'ezout equation in\\
the ring of periodic distributions}

\author[R. Rupp]{Rudolf Rupp}
\address{Fakult\"at f\"ur Angewandte Mathematik, Physik  und Allgemeinwissenschaften\\
 TH-N\"urnberg, Kesslerplatz 12\\ D-90489 N\"urnberg, Germany}
\email{Rudolf.Rupp@th-nuernberg.de}

\author[A. Sasane]{Amol Sasane}
\address{Department of Mathematics \\London School of Economics\\
     Houghton Street\\ London WC2A 2AE\\ United Kingdom}
\email{sasane@lse.ac.uk}

\begin{abstract}
  A corona type theorem is given for the ring $\calD'_\mA(\mR^d)$ of
  periodic distributions in $\mR^d$ in terms of the sequence of
  Fourier coefficients of these distributions, which have at most
  polynomial growth.  It is also shown that the Bass stable rank and
  the topological stable rank of $\calD'_\mA(\mR^d)$ are both equal to
  $1$.
\end{abstract}

\maketitle

\section{Introduction}

The aim of this short note is to study some algebraic and topological
questions associated with the ``B\'ezout equation''
$$
\bbb_1 \bba_1 +\cdots+ \bbb_N\bba_N=\bbe,
$$
where $\bbb_i, \bba_i$ ($1\leq i\leq N$) are elements of the
commutative unital topological ring
$$
(\calD'_\mA(\mR^d),+,\ast,\calT_{\calD'_\mA(\mR^d)}),
$$
defined below, and $\bbe$ denotes the identity element (which will be
the locally finite sum of Dirac distributions placed at a lattice
formed by the periods in $\mA$, as explained below). The B\'ezout
equation in rings of distributions arises in problems of robust
filtering, image processing, etc, see for example \cite{Bos}.
 
\subsection{The spaces $\calD'_\mA(\mR^d)$ and $s'(\mZ^d)$} 

For background on periodic distributions and its Fourier series
theory, we refer the reader to \cite[Chapter~16]{Dui} and
\cite[p.527-529]{Tre}.

Consider the space $s'(\mZ^d)$ of all complex valued maps on $\mZ^d$
of at most polynomial growth, that is,
$$
s'(\mZ^d):=\{\mathbf{a}: \mZ^d\rightarrow \mC\;|\;
\exists M>0 : \exists k\in \mN : 
\forall \bbn\in \mZ^d: 
|\mathbf{a}(\bbn)|\leq M (1+\nm \bbn\nm)^k\},
$$
where $\nm \cdot\nm$ denotes the $1$-norm in $\mR^d$. Then this is a
unital commutative ring $(s'(\mZ^d), +,\cdot)$ with pointwise
operations $+$ and $\cdot$, and the unit element being the constant
function
$$
\bbn\mapsto 1\quad (\bbn\in \mZ^d).
$$ 
We now equip it with a topology $\calT_{s'(\mZ^d)}$, as follows.
First, consider the locally convex topological vector space $s(\mZ^d)$
of rapidly decreasing sequences:
$$
s(\mZ^d)
:=
\Big\{\mathbf{a}: \mZ^d\rightarrow \mC\;\Big|\;
\forall k\in \mN, 
\;\sup_{\bbn\in \mZ^d} (1+\nm \bbn\nm)^k |\bba(\bbn)|<+\infty
\Big\},
$$
with pointwise operations, and with the topology given by the family
of seminorms
$$
p_k(\bbb):= 
\sup_{\bbn \in \mZ^d} (1+\nm \bbn \nm)^k |\bbb(\bbn)|,\quad 
\bbb\in s(\mZ^d),\quad k\in \mN.
$$
(That is, the topology on $s(\mZ^d)$ is the weakest one making all the
seminorms continuous.)  Note that there is a natural duality between
$s'(\mZ^d)$ and $s(\mZ^d)$, namely a bilinear form on $s'(\mZ^d)\times
s(\mZ^d) $ defined as follows: for $\bba\in s'(\mZ^d)$ and $\bbb\in
s(\mZ^d)$,
$$
\langle \bba,\bbb\rangle_{s'(\mZ^d)\times s(\mZ^d)}
:=\sum_{\bbn\in \mZ^d} \bba(\bbn)\bbb(\bbn) .
$$
Equip $s'(\mZ^d)$ by its natural weak-$\ast$ topology
$\calT_{s'(\mZ^d)}$ as a dual of $s(\mZ^d)$.  This topology
$\calT_{s'(\mZ^d)}$ can be described in terms of convergence of nets
as follows: a net $(\bba_i)_{i\in I}$ in $s'(\mZ^d)$ converges to
$\bba$ in $s'(\mZ^d)$ if and only if  for every $\bbb\in s(\mZ^d)$, we have that
$$
\lim_{i} \langle \bba_i,\bbb\rangle =\langle \bba, \bbb\rangle.
$$
Then $s'(\mZ^d)$ equipped with pointwise operations, and the above
topology $\calT_{s'(\mZ^d)}$, is a topological ring.  Moreover, this
is isomorphic as a topological ring to $\calD'_\mA(\mR^d)$ ($\simeq
\calD'(\mT^d)$), where the latter is equipped with pointwise addition
and multiplication taken as convolution, and its natural dual topology
$\calT_{\calD'_\mA(\mR^d)}$, as elaborated below.

For ${\mathbf{a}}\in {{\mathbb{R}}}^{d}$, the {\em translation
  operation} ${\mathbf{S_a}}$ on distributions in
${\mathcal{D}}'({\mathbb{R}}^d)$ is defined by
$$
\langle {\mathbf{S_a}}(T),\varphi\rangle
=
\langle T,\varphi(\cdot+{\mathbf{a}})\rangle\;
\textrm{ for all }\varphi \in {\mathcal{D}}({\mathbb{R}}^d).
$$
A distribution $T\in {\mathcal{D}}'({\mathbb{R}}^d)$ is said to be
{\em periodic with a period} $\mathbf{a}\in {\mathbb{R}}^d\setminus
\{\mathbf{0}\}$ if $T= {\mathbf{S_a}}(T)$.  Let
$$
{\mathbb{A}}:=\{{\mathbf{a_1}}, \cdots, {\mathbf{a_d}}\}
$$ 
be a linearly independent set of $d$ vectors in ${\mathbb{R}}^d$.  We
define ${\mathcal{D}}'_{{\mathbb{A}}}({\mathbb{R}}^d)$ to be the set
of all distributions $T$ that satisfy
$$
{\mathbf{S_{a_k}}}(T)=T \quad ( k=1,\cdots, d).
$$
From \cite[\S34]{Don}, $T$ is a tempered distribution, and from the
above it follows by taking Fourier transforms that
$$
(1-e^{2\pi i {\mathbf{a_k}} \cdot \mathbf{y}})\widehat{T}=0
\textrm{ for } k=1,\cdots,d.
$$ 
It can be seen that 
$$
\widehat{T}
=
\sum_{\mathbf{v} \in A^{-1} {\mathbb{Z}}^d} 
\alpha_{\mathbf{v}}(T) \delta_{\mathbf{v}},
$$
for some scalars $\alpha_{\mathbf{v}}(T)\in {\mathbb{C}}$, and where
$A$ is the matrix with its rows equal to the transposes of the column
vectors ${\mathbf{a_1}}, \cdots, {\mathbf{a_d}}$:
$$
A:= \left[ \begin{array}{ccc} 
    \mathbf{a_1}^{\top} \\ \vdots \\ \mathbf{a_d}^{\top} 
    \end{array}\right].
$$
Also, in the above, $\delta_{\mathbf{v}}$ denotes the usual Dirac
measure with support in $\mathbf{v}$:
$$
\langle \delta_{\mathbf{v}}, \psi\rangle 
=\psi (\mathbf{v}) 
\;\textrm{ for }\psi \in {\mathcal{D}}({\mathbb{R}}^d).
$$
Then the Fourier coefficients $\alpha_{ \bbv}(T)$ give rise to an
element in $s'(\mZ^d)$, and vice versa, every element in $s'(\mZ^d)$
is the set of Fourier coefficients of some periodic distribution. In
this manner the topological ring
$$
(\calD'_{\mA}(\mR^d),+,\ast,\calT_{\calD'_{\mA}(\mR^d)})
$$
of periodic distributions on $\mR^d$ is isomorphic (as topological
rings) to
$$
(s'(\mZ^d),+,\cdot,\calT_{s'(\mZ^d)}).
$$
In the sequel we will mostly just consider the topological ring
$$
(s'(\mZ^d),+,\cdot,\calT_{s'(\mZ^d)})
$$ 
while stating our results and
demonstrating proofs, with the tacit understanding that analogous
results also hold for
 $(\calD'_{\mA}(\mR^d),+,\ast,\calT_{\calD'_{\mA}(\mR^d)})$.

\subsection{Main results} 

We will prove three results (Theorem~\ref{theorem_corona},
\ref{theorem_bsr} and \ref{theorem_tsr}), listed below:

\begin{theorem}
\label{theorem_corona}
Let $\mathbf{a}_1,\cdots, \mathbf{a}_N\in s'(\mZ^d)$. Then the
following are equivalent:
\begin{enumerate}
\item There exist $\mathbf{b}_1,\cdots, \mathbf{b}_N\in s'(\mZ^d)$
  such that
 $$
 \mathbf{b}_1 \mathbf{a}_1 +\cdots +\mathbf{b}_N \mathbf{a}_N =1
 $$
\item There exists a $\delta >0$ and a $K\in \mN$ such that
 $$
 \forall \bbn\in \mZ^d, \; 
 |\mathbf{a}_1(\bbn)|+\cdots+  |\mathbf{a}_N(\bbn)|
 \geq \delta (1+\nm \bbn \nm)^{-K}.
 $$
\end{enumerate}
\end{theorem}

\noindent In light of the isomorphism between $\calD'_\mA(\mR^d)$ and
$s'(\mZ^d)$, this result can be viewed as an analogue of a result for
the solvability of the B\'ezout equation in the ring $\calE'(\mR^d)$
of compactly supported distributions recalled below (see for instance
\cite[Corollary~3.1]{MS}):

\begin{proposition}
\label{prop_MS_IEOT}
Let $f_1,\cdots ,f_N\in \calE'(\mR^d)$. Then the following are
equivalent:
\begin{enumerate}
\item There exist $g_1,\cdots ,g_N\in \calE'(\mR^d)$ such that
$$
f_1 \ast g_1 +\cdots + f_N \ast  g_N
=
\delta.
$$
\item There are positive constants $\delta , K, M$ such that for all
  $\bbz \in \mC^d$,
$$
|\widehat{f}_1(\bbz)|+\cdots + |\widehat{f}_N(\bbz)|\geq 
\delta (1 +\nm \bbz\nm)^{-K}e^{-M\nm \textrm{\em Im}(\bbz)\nm }.
$$
\end{enumerate}
\end{proposition}

\noindent Our second main result is the determination of the Bass
stable rank of $s'(\mZ^d)$. This notion of stable rank was introduced
by Hyman Bass in order to facilitate some computations in algebraic
$K$-theory where it plays a role analogous to dimension; see
\cite{Bas},

\begin{definition}[Bass stable rank]
  Let $R$ be a commutative unital ring with identity element $1$. We
  assume that $1\not=0$, that is $R$ is not the trivial ring $\{0\}$.

\begin{enumerate}
\item An $N$-tuple $(\bba_1,\cdots,\bba_N)\in R^N$ is said to be {\it
    invertible} (or {\it unimodular}), if there exists
  $(\bbb_1,\cdots,\bbb_N)\in R^N$ such that the B\'ezout equation
 $$
  \bbb_1 \bba_1+\cdots+ \bbb_N\bba_N =1
 $$
 is satisfied.
   The set of all invertible $N$-tuples is denoted by $U_N(R)$. 
   
 \item An $(N+1)$-tuple $(\bba_1,\cdots,\bba_N,\balpha)\in U_{N+1}(R)$
   is called {\it reducible} if there exists
   $(\bbh_1,\cdots,\bbh_N)\in R^N$ such that
 $$
 (\bba_1+\bbh_1 \balpha,\cdots, \bba_N+\bbh_N\balpha )\in U_N(R).
 $$

\item The {\it Bass stable rank} of $R$, denoted by $\textrm{bsr }R$,
  is the smallest integer $N$ such that every element in $U_{N+1}(R)$
  is reducible.  If no such $N$ exists, then $\textrm{bsr }R:=\infty$.
\end{enumerate}
\end{definition}

\noindent Here is the reason for taking the smallest such number $N$: 
if every $(N+1)$-tuple is reducible, then also every $(N+k)$-tuple, 
$k \geq 1$, is reducible,  see \cite{Bas}. With this terminology, our second main 
result is the following:
  
\begin{theorem}
\label{theorem_bsr}
The Bass stable rank of $s'(\mZ^d)$ is $1$.
\end{theorem}

As an immediate consequence of this result, we also have that $s'(\mZ^d)$ is a 
Hermite ring, as elaborated below. 

\begin{definition}[Hermite ring]
A  commutative unital ring $R$ is called a {\em Hermite ring} if for every $N\in \mN$, 
and every $ (\bba_1,\cdots, \bba_N )\in U_N(R)$, there exists a $N\times N$ matrix 
$\bbA\in R^{N\times N}$, which is invertible as an element of $R^{N\times N}$, and such that 
 its first column entries $\bbA_{i1}$ coincide with $\bba_i$:
 $$
 \bbA_{i1}=\bba_i\quad (i=1,\cdots, N).
 $$
\end{definition}

\begin{corollary}
 $s'(\mZ^d)$ is a Hermite ring. 
\end{corollary}
\begin{proof} This follows from the known result that commutative unital rings having 
stable rank at most $2$ are Hermite (see for example \cite[Theorem~20.13, p. 301]{Lam}). 
\end{proof}

\noindent Finally, we will also determine the topological stable rank
of the topological ring $s'(\mZ^d)$. Marc~Rieffel, in the seminal paper \cite{Rie},
introduced a notion of topological stable rank, analogous to the
concept of Bass stable rank, for $C^*$-algebras. Motivated by his
definition, one can more generally consider topological rings
instead of only $C^*$-algebras, and consider the following
generalization of his notion of topological stable rank.
 
\begin{definition}[Topological stable rank] 
  Let $R$ be a commutative unital topological ring with topology
  $\mathcal T$.  The {\it topological stable rank}, ${\rm
    tsr }R$, of $(R,\mathcal T)$ is the least integer $N$
  for which $U_N(R)$ is dense in $R^N$. If no such $N$ exists, then 
  $\textrm{tsr }R:=\infty$.
\end{definition}

\noindent Our third and final main result is:
  
\begin{theorem}
\label{theorem_tsr}
The topological stable rank of $s'(\mZ^d)$ is $1$. 
\end{theorem}
  
\noindent The organization of the paper is as follows: In
Sections~\ref{section_corona}, \ref{section_bsr}, \ref{section_tsr},
we give the proofs of Theorems~\ref{theorem_corona},
\ref{theorem_bsr}, \ref{theorem_tsr}, respectively.

\section{Corona type theorem}
\label{section_corona}

\begin{proof}[Proof of Theorem~\ref{theorem_corona}] $\;$

\noindent (1)$\Rightarrow$(2): Suppose there exist
  $\mathbf{b}_1,\cdots, \mathbf{b}_N\in s'(\mZ^d)$ such that
 $$
 \mathbf{b}_1 \mathbf{a}_1 +\cdots +\mathbf{b}_N \mathbf{a}_N =1. 
 $$
 Then we can choose a large enough $K\in \mN$ and $M>0$ such that for
 all $\bbn \in \mZ^d$ and $i=1,\cdots,N$,
 $$
 |\bbb_i(\bbn)|\leq M(1+\nm \bbn \nm)^K.
 $$
 Then for all $\bbn\in \mZ^d$, we have 
\begin{eqnarray*}
1
&=&
|1|=
|\mathbf{b}_1(\bbn) \mathbf{a}_1(\bbn) +\cdots +\mathbf{b}_N(\bbn) \mathbf{a}_N(\bbn)|
\\
&\leq & 
|\mathbf{b}_1(\bbn) ||\mathbf{a}_1(\bbn)| +\cdots +|\mathbf{b}_N(\bbn)|| \mathbf{a}_N(\bbn)|
\\
&\leq & 
M(1+\nm \bbn \nm)^K (|\mathbf{a}_1(\bbn)| +\cdots +| \mathbf{a}_N(\bbn)|),
\end{eqnarray*}
and so the corona type condition follows by a rearrangement (and with
$\delta:=1/M$).

\medskip 

\noindent (2)$\Rightarrow$(1): From the corona type condition, it
follows that
$$
\bbn\mapsto \frac{1}{ |\mathbf{a}_1(\bbn)|+\cdots+  |\mathbf{a}_N(\bbn)|}
$$
is an element of $s'(\mZ^d)$. Next, for $i=1,\cdots,N$ and $\bbn\in
\mZ^d$, define
$$
\mathbf{b}_i(\bbn):=
\frac{e^{-i\textrm{Arg}(\mathbf{a}_i(\bbn))}}{ |\mathbf{a}_1(\bbn)|+\cdots+  |\mathbf{a}_N(\bbn)|},
$$
where $\textrm{Arg}$ denotes the principal argument of a complex
number, living in $(-\pi,\pi]$, and is taken as $0$ for the complex
number $0$.  Then we see that $\mathbf{b}_1,\cdots, \mathbf{b}_N\in
s'(\mZ^d)$ and
$$
\mathbf{b}_1 \mathbf{a}_1 +\cdots +\mathbf{b}_N \mathbf{a}_N =1
$$
This completes the proof.
\end{proof}

\section{Bass stable rank of $s'(\mZ^d)$} 
\label{section_bsr}

\begin{proof}[Proof of Theorem~\ref{theorem_bsr}] Suppose that
  $\mathbf{a}_1,\mathbf{a}_2,\mathbf{b}_1,\mathbf{b}_2$ are elements
  of $s'(\mZ^d)$ such that
$$
\mathbf{b}_1\mathbf{a}_1+\mathbf{b}_2\mathbf{a}_2=1.
$$
Define $\mathbf{u}_1\in s'(\mZ^d)$ by
$\mathbf{u}_1(\bbn)=1+|\bba_1(\bbn)|$, $\bbn\in \mZ^d$. Then by using
Theorem~\ref{theorem_corona}, we can see that $\bbu_1$ is invertible
in $s'(\mZ^d)$. Set $\bbA_1:=\bba_1 \bbu_1^{-1}$, that is,
$$
\bbA_1(\bbn)=\frac{\bba_1(\bbn)}{1+|\bba_1(\bbn)|}\quad (\bbn\in \mZ^d).
$$
We note that $|\bbA_1(\bbn)|\leq 1$ for all $\bbn\in \mZ^d$. Define
$\bbB_1\in s'(\mZ^d)$ by $\bbB_1=\bbb_1\bbu_1$. Then $\bbb_1
\bba_1=\bbB_1 \bbA_1$, and so
$$
1=\bbb_1 \bba_1+\bbb_2\bba_2=\bbB_1\bbA_1+\bbb_2\bba_2.
$$
Define $\widetilde{\bbB}_1$ by 
$$
\widetilde{\bbB}_1(\bbn)
=\left\{\begin{array}{ll} 
 \epsilon & \textrm{if } |\bbB_1(\bbn)|\leq \epsilon,\\
 \bbB_1(\bbn) & \textrm{if } |\bbB_1(\bbn)|> \epsilon,
\end{array}\right.
$$
where the $\epsilon>0$ will be determined later. Then clearly
$\widetilde{\bbB}_1\in s'(\mZ^d)$, and in fact is invertible in
$s'(\mZ^d)$ because it is bounded below by $\epsilon$.  Moreover,
$\widetilde{\bbB}_1$ ``approximates'' $\bbB_1$ pointwise:
$$
|\widetilde{\bbB}_1(\bbn)-\bbB_1(\bbn)|
\leq 2\epsilon \quad (\bbn\in \mZ^d).
$$
We also note that 
$$
\widetilde{\bbB}_1 \bbA_1 + \bbb_2 \bba_2
= \widetilde{\bbB}_1 \bbA_1 +(1-\bbB_1 \bbA_1)
=1+(\widetilde{\bbB}_1-\bbB_1)\bbA_1.
$$
Since for all  $\bbn\in \mZ^d $ we have that 
$$
|(1+(\widetilde{\bbB}_1-\bbB_1)\bbA_1)(\bbn)|
\geq 
1-2\epsilon \cdot 1 
= 1-2\epsilon \stackrel{(\epsilon:=1/4)}{=} 
\frac{1}{2},
$$
for the choice $\epsilon=1/4$, it follows by
Theorem~\ref{theorem_corona}, that
$1+(\widetilde{\bbB}_1-\bbB_1)\bbA_1$ is an invertible element in
$s'(\mZ^d)$.

Now with $\bbh\in s'(\mZ^d)$ defined by $ \bbh :=
\widetilde{\bbB}_1^{-1} \bbu_1 \bbb_2 , $ we have
\begin{eqnarray*}
\bba_1 +\bbh \bba_2 
&=&
\widetilde{\bbB}_1^{-1}\bbu_1 
( \widetilde{\bbB}_1 \bbu_1^{-1} \bba_1 +    \bbb_2\bba_2 )
\\
&=&
\widetilde{\bbB}_1^{-1}\bbu_1 
(\widetilde{\bbB}_1\bbA_1+    \bbb_2\bba_2 )
\\
&=&
\widetilde{\bbB}_1^{-1}\bbu_1 
(1+(\widetilde{\bbB}_1-\bbB_1)\bbA_1).
\end{eqnarray*}
As the last expression on the right hand side is a product of
invertibles from $s'(\mZ^d)$, it follows that $\bba_1 +\bbh \bba_2 $
is an invertible in element in $s'(\mZ^d)$. This shows that the Bass
stable rank of $s'(\mZ^d)$ is $1$.
\end{proof}

\begin{remark} 
In fact, a modification of the proof gives a more general result, 
as outlined below.
 
Let $c(\mZ^d)$ denote the ring of all complex valued maps on $\mZ^d$
with the usual pointwise operations, and $\ell^\infty (\mZ^d)$ denote
the subring of $c(\mZ^d)$ consisting of all maps with a bounded
range. Suppose that $R$ is a subring of $c(\mZ^d)$ possessing the
following properties (P1), (P3) and either (P2) or (P$2^\prime$):
\begin{itemize}
\item[(P1)] $\ell^\infty(\mZ^d) \subset R$, 
\item[(P2)] $f\in R \;\Rightarrow \;\overline{f}\in R$. (Here
  $\overline{\;\cdot\;}$ indicates pointwise complex conjugation.)
\item[(P$2^\prime$)] $f\in R\;\Rightarrow \;|f|\in R$.  (Here $|\cdot|$
  denotes taking pointwise complex absolute value.)
\item[(P3)] $|f|\geq \delta>0 \;\Rightarrow \;f \in U_1(R)$.
\end{itemize}
Note that once we have (P1), the properties (P2) and (P$2^\prime$) are
equivalent, thanks to the identity $
\overline{z}e^{i\textrm{Arg}(z)}=|z|
$ 
for any complex number $z$, where $\textrm{Arg}$ denotes the principal
argument of a complex number, living in $(-\pi,\pi]$, and we set $\textrm{Arg}(0):=0$.
  
\smallskip 
  
\noindent {\bf Claim}: $\textrm{bsr}(R)=1$. 
  
\smallskip 
  
\noindent We carry out the same proof as above, and note the following
 key changes in the arguments:
\begin{enumerate}
\item $\bbu_1 \in R$ because $|\bba_1| \in R$, using (P2) or
  (P$2^\prime$). 
 Moreover (P3) shows also that $\bbu_1$ is invertible
  in $R$. Also, then $\bbA_1 \in R $ and $\bbB_1 \in R$. 
\item Let's show that $\widetilde{\bbB}_1\in R$. To this end, write additively
  $\widetilde{\bbB}_1=\bbx + \bbB_1\cdot \bby$, where
  $$
  \quad \quad \bbx(\bbn)\!:=\! 
  \left\{ \!\!\begin{array}{ll} 
          \epsilon &\!\textrm{if } |\bbB_1(\bbn)|\leq \epsilon, \\
           0 &\!\textrm{if } |\bbB_1(\bbn)|> \epsilon 
  \end{array}\!\!\right\} 
  \textrm{ and } 
  \bby(\bbn)\!:=\! 
  \left\{ \!\!\begin{array}{ll} 
          0 &\!\textrm{if } |\bbB_1(\bbn)|\leq \epsilon, \\
          1 &\!\textrm{if } |\bbB_1(\bbn)|> \epsilon 
  \end{array}\!\!\right\} .
  $$
  As $\bbx,\bby$ are bounded, they belong to $R$ by (P1). Thus, since $R$ is a ring, so
  does $\widetilde{\bbB}_1=\bbx + \bbB_1\cdot \bby$.
\item We prove that $\widetilde{\bbB}_1$ is invertible in
  $R$.  To this end, factor multiplicatively
  $\widetilde{\bbB}_1=(1+|\widetilde{\bbB}_1|)\cdot \bbC_1$, where
  $$
  \bbC_1=\frac{\widetilde{\bbB}_1}{1+|\widetilde{\bbB}_1|\phantom{\displaystyle \sum}\!\!\!\!\!\!\!\!\!}.
  $$
  The first factor, $1+|\widetilde{\bbB}_1|$, belongs to $R$, using (P2)
  and the fact that $\widetilde{\bbB}_1\in R$ (shown above). 
  Moreover, $1+|\widetilde{\bbB}_1|$ is invertible in $R$ by (P3).
  
  \noindent Now the function
   $
  x \mapsto \displaystyle \frac{x}{1+x}:(0,\infty)\rightarrow (0,\infty)
  $ 
  is  increasing, and 
  so 
  $$
  \frac{\epsilon}{1+\epsilon} \leq |\bbC_1|\leq 1.
  $$
  The rightmost inequality shows that $\bbC_1\in R$ (as it is bounded), while the 
  leftmost inequality then gives the invertibility of $\bbC_1$ in
  $R$ (using (P3)). The rest of the proof is the same, mutatis
  mutandis, as the proof of Theorem~\ref{theorem_bsr}.
\end{enumerate}
\end{remark}

\section{Topological stable rank of $s'(\mZ^d)$} 
\label{section_tsr}

\begin{proof}[Proof of Theorem~\ref{theorem_tsr}]
  Let $\bba\in s'(\mZ^d)$. We will construct a net
  $(\bba_\epsilon)_{\epsilon>0}$ with index set as the directed set
  $(0,\infty)$ and the usual order of real numbers, of invertible
  elements $\bba_\epsilon$ in $s'(\mZ^d)$ such that
  $(\bba_\epsilon)_{\epsilon>0}$ converges to $\bba$.  Define for
  $\epsilon>0$ and $\bbn\in \mZ^d$
 $$
 \bba_\epsilon(\bbn)
 =\left\{\begin{array}{ll} 
   \epsilon & \textrm{if }|\bba(n)|\leq \epsilon,\\
   \bba(\bbn) & \textrm{if }|\bba(n)|>\epsilon.
   \end{array}\right.
 $$
 Then $|\bba_\epsilon(\bbn)|\geq \epsilon$ for all $\bbn\in \mZ^d$,
 and so $\bba_\epsilon$ is invertible in $s'(\mZ^d)$ for all
 $\epsilon>0$. Moreover, for every $\bbb\in s(\mZ^d)$, we have
 \begin{eqnarray*}
 \Big|\langle (\bba_\epsilon-\bba),\bbb\rangle_{s'(\mZ^d)\times s(\mZ^d)}\Big|
 &=& 
 \Big|\sum_{\bbn\in \mZ^d} (  \bba_\epsilon-\bba)(\bbn)\bbb(\bbn) \Big|
 \\
 &\leq & 
 \sum_{\bbn\in \mZ^d} |\bba_\epsilon(\bbn)-\bba(\bbn)|\cdot |\bbb(\bbn) |
 \\
 & \leq & \sum_{\bbn\in \mZ^d} 
 2\epsilon  \cdot |\bbb(\bbn) |
 \\
 &\leq & 2\epsilon \sum_{\bbn\in \mZ^d} \frac{K}{(1+\nm \bbn \nm)^2}\leq K' \epsilon,
 \end{eqnarray*}
 where $K$ (and $K'$) is a constant (depending on $\bbb$). So the net
 $(\bba_\epsilon)_{\epsilon>0}$ of invertible elements $\bba_\epsilon$
 in $s'(\mZ^d)$ converges in the weak-$\ast$ topology of $s'(\mZ^d)$
 to $\bba$.  Hence $U_1(s'(\mZ^d))$ is dense in $s'(\mZ^d)$, that is,
 the topological stable rank of $ s'(\mZ^d)$ is $1$.
\end{proof}

\begin{remark}
  It is known that for a commutative $Q$-algebra (namely a unital
  topological algebra in which the set of units forms an open set)
  $R$, one has $\textrm{bsr}(R)\leq \textrm{tsr}(R)$ (see for example 
  the proof of \cite[Theorem~2.3]{Rie}). However, it can
  be seen that $s'(\mZ^d)$ is {\em not} a $Q$-algebra (see below), and
  so the result in this section does not render the result in the
  previous section superfluous.
 
  Let us show that the set of units in $s'(\mZ^d)$ does not form an open
  set in the weak-$\ast$ topology of $s'(\mZ^d)$.  Consider the net
  $(\bba_\epsilon)_{\epsilon>0}$, where
  $$
  \bba_\epsilon:=e^{-\epsilon |\bbn|} \quad (\bbn\in \mZ^d). 
  $$
  Then each $\bba_\epsilon$, $\epsilon>0$, is not invertible in
  $s'(\mZ^d)$ because it is easily seen that
  $$
  \neg \Big(\exists \delta>0:\exists K\in \mN:|\bba_\epsilon(\bbn)|\geq \delta (1+\nm\bbn\nm)^{-K}\Big).
  $$
  But the net $(\bba_\epsilon)_{\epsilon>0}$ converges to the
  invertible element $(1)_{\bbn\in \mZ^d}$: indeed for every $\bbb\in
  s(\mZ^d)$, we have
  \begin{eqnarray*}
  | \langle (\bba_\epsilon-\bba),\bbb\rangle_{s'(\mZ^d)\times s(\mZ^d)} |
  &=& 
  \Big| \sum_{\bbn\in \mZ^d} (1-e^{-\epsilon \nm \bbn \nm}) \bbb(\bbn)\Big|
  \\
  &\leq &  \sum_{\bbn\in \mZ^d} \epsilon \nm n\nm |\bbb(\bbn)|\\
  &\leq &\epsilon \sum_{\bbn\in \mZ^d} \nm \bbn\nm \cdot\frac{K}{(1+\nm \bbn\nm)^3}
  =K'\cdot \epsilon,
  \end{eqnarray*}
  for some constant $K$ (and $K'$) depending on $\bbb$. In the above, we have used the
  inequality $0<1-e^{- x}<x$ for $x>0$, which follows from the
  Mean Value Theorem.
\end{remark}

\section{An open problem}

While we have computed the 
 stable ranks of $s'(\mZ^d)\simeq \calD_\mA'(\mR^d)$, 
  the stable ranks of $\calE'(\mR^d)$ do not seem to be known. Via Fourier transformation, 
 and by the Payley-Wiener-Schwartz Theorem \cite[Theorem~7.3.1, p.181]{Ho2},   
  $(\calE'(\mR^d),+,\ast)$  is isomorphic to the ring 
 $E_{\textrm{exp}}$ of entire functions of exponential type:
 $$
 E_{\textrm{exp}}(\mC^d)\!:=\!\Big\{\!f:\!\mC^d\!\rightarrow \mC: 
 \exists C>0 :\!\exists N\in \mN:\!\exists\! M>0:\!|f(\bbz)|\leq \!C
 e^{M\nm \textrm{Im}(\bbz)\nm }\Big\},
$$
equipped with pointwise operations.  Is $\textrm{bsr}(E_{\textrm{exp}}(\mC^d))<\infty$? If so what is it? 
The Bass stable rank of $E_{\textrm{exp}}(\mC)$ can be shown to 
be at least $2$ (see below).

First, we note that every unit in $E_{\textrm{exp}}(\mC)$ has the form 
 $u(z)=e^{a+bz}$ for some constants $a,b$.  This follows directly from the 
 Weierstrass factorisation of such functions, but here is a direct proof:
 writing $u(z)=e^{h(z)}$, where $h$ denotes an entire function, we conclude that $|\textrm{Re}(h(z))|\leq A+B|z|$, 
  but then Schwarz's formula for the disc $|\zeta |=2r$ and $|z|=r$ 
   gives 
   $$
   h(z)=\frac{1}{2\pi i}
   \int_{|\zeta|=2r}\frac{\zeta+z}{\zeta-z} \textrm{Re}(h(\zeta))\frac{ d\zeta}{\zeta} +i\textrm{Im}(h(0))
   $$
   and so, 
   $$
   |h(z)-i\textrm{Im}(h(0))|\leq \frac{3r}{r}\sup_{|\zeta|=2r} |\textrm{Re}(h(\zeta))|\leq 3(A+2rB).
   $$
   Now Cauchy's estimate for the coefficients of a Taylor series 
   shows that $h$ must be a polynomial of degree 1, that is $h(z)=a+bz$ 
   for some $a,b$. 
   
   Finally, let us show that $\textrm{bsr}(E_{\textrm{exp}}(\mC))>1$. Note that 
   $(z-1,z^3)$ is a unimodular pair in 
   $E_{\textrm{exp}}(\mC)$ (with a polynomial solution to the B\'ezout equation),  
   but is not reducible. Assuming the contrary, there are $h\in E_{\textrm{exp}}(\mC)$ 
   and constants $a,b$ such that $z-1+h(z)z^3=e^{a+bz}$. 
   Differentiating twice gives:
$$
h''(z)z^3+6h'(z)z^2+ h(z)6z=b^2 e^{a+bz},
$$
and putting $z=0$ gives $b=0$. So $z-1+h(z)z^3=e^a$ for all $z$. Differentiating this 
gives $1+h'(z) z^3+h(z) 3z^2=0$, and putting $z=0$ now gives the contradiction that $1=0$.

\end{document}